\documentclass[twoside,11pt,reqno]{amsart}
\usepackage{amsmath,amssymb,amscd,mathrsfs,amscd, todonotes}
\usepackage{graphics,verbatim}
\usepackage{todonotes}
\usepackage{enumitem}
\usepackage{hyperref}

\newcommand{\losemi}{{\otimes \kern -.78em \ltimes}}
\newcommand{\rosemi}{{\otimes \kern -.78em \rtimes}}

\newcommand{\Dist}[1]{\operatorname{Dist}(#1)}

\newcommand{\rank}{\operatorname{rank}}

\newcommand{\St}{\operatorname{St}}

\makeatletter
\newcommand{\leqnomode}{\tagsleft@true}
\newcommand{\reqnomode}{\tagsleft@false}
\makeatother

\newtheorem{theorem}{Theorem}[subsection]

\makeatletter\let\c@fact\c@theorem\makeatother

\makeatletter\let\c@note\c@theorem\makeatother

\newtheorem{lemma}{Lemma}[subsection]
\makeatletter\let\c@lemma\c@theorem\makeatother

\makeatletter\let\c@lemma\c@theorem\makeatother

\makeatletter\let\c@alg\c@theorem\makeatother

\makeatletter\let\c@prop\c@theorem\makeatother

\makeatletter\let\c@conj\c@theorem\makeatother

\newtheorem{cor}{Corollary}[subsection]
\makeatletter\let\c@cor\c@theorem\makeatother

\makeatletter\let\c@defn\c@theorem\makeatother

\theoremstyle{definition}

\newtheorem{remark}{Remark}[subsection]
\makeatletter\let\c@remark\c@theorem\makeatother

\makeatletter\let\c@example\c@theorem\makeatother
\numberwithin{equation}{subsection}

\usepackage[capitalise]{cleveref}
\crefname{theorem}{Theorem}{Theorems}
\crefname{fact}{Fact}{Facts}
\crefname{note}{Note}{Notes}
\crefname{lemma}{Lemma}{Lemmas}
\crefname{alg}{Algorithm}{Algorithms}
\crefname{remark}{Remark}{Remarks}
\crefname{example}{Example}{Examples}
\crefname{prop}{Proposition}{Propositions}
\crefname{conj}{Conjecture}{Conjectures}
\crefname{cor}{Corollary}{Corollaries}
\crefname{defn}{Definition}{Definitions}
\crefname{equation}{\!\!}{\!\!} 

\newcounter{listequation}

\begin{document}

\title{Steinberg Squares and tensor products of\\ Tilting modules with simple Modules}

\author{Paul Sobaje}
\address{Department of Mathematics \\
          Georgia Southern University}
\email{psobaje@georgiasouthern.edu}
\date{\today}

\begin{abstract}
Let $G$ be a simple and simply connected algebraic group over an algebraically closed field $\Bbbk$ of characteristic $p>0$.  We establish an isomorphism of $G$-modules between a direct sum of modules $\St \otimes \St$ and a direct sum of tensor products of simple modules of restricted highest weight with tilting modules that are projective over the Frobenius kernel of $G$.  This isomorphism holds precisely when Donkin's Tilting Module Conjecture does, and thus can be seen as providing a $G$-module theoretic characterization of this conjecture.
\end{abstract}

\maketitle

\section{Introduction}

Let $G$ be a simple and simply connected algebraic group over an algebraically closed field $\Bbbk$ of characteristic $p>0$.  The Steinberg module $\St$ is the simple $G$-module $L((p-1)\rho)$, where $\rho$ is the sum of the fundamental dominant weights.  It is a tilting module, is self-dual, and is both simple and projective over the Frobenius kernel $G_1$.  It is the most important example of a projective indecomposable $G_1$-module that lifts to a $G$-module, and has been recognized since the earliest work on the topic as holding the key to determining which other projective indecomposable $G_1$-modules carry a $G$-structure.

What has become clear in recent years is that the modules $\St \otimes L(\lambda)$, with $\lambda \in X_1(T)$ (the set of $p$-restricted dominant weights), contain even more information regarding this problem than first suspected.  Specifically, the projective indecomposable $G_1$-modules lift to tilting modules for $G$ only if $\St \otimes L(\lambda)$ is tilting \cite{KN}.  Pairing this with the fact that tilting modules have become main players in the character theory of $G$ (\cite{RW}, \cite{AMRW}), we find that understanding the complete decompositions of these modules, over $G$ and over $G_1$, is more important than ever.

The highest weight summand of $\St \otimes L(\lambda)$ is the tilting module $T((p-1)\rho+\lambda)$ \cite{P}.  One can also compute the character of the component of this module that lies in the $G_1$-block of the Steinberg module.  And while further analysis has been carried out in \cite{BNPS} (see \cite{K} for a related setting), there is no known method of determining all of the summands of these modules.  One might hope that $\St \otimes \St$ (the ``Steinberg square'') would be more reasonable to work with since the character of this module is easily computed, but this case seems to be as hard as any of the others.

The main result of this note is to show that a direct sum of Steinberg squares can be described, in a sense, by proving it to be isomorphic to a direct sum of tensor products between $G_1$-projective tilting modules and simple $G$-modules of restricted highest weight.  This isomorphism holds precisely when Donkin's Tilting Module Conjecture does, and thus can be seen as providing a $G$-module theoretic characterization of this conjecture (i.e. without reference to $G_1$).

\begin{theorem}\label{T:A}
There is an isomorphism of $G$-modules
$$(\St \otimes \St)^{\oplus p^{\rank(G)}} \cong \bigoplus_{\lambda \in X_1} \left(T((p-1)\rho+\lambda) \otimes L((p-1)\rho-\lambda) \right)$$
if and only if $T((p-1)\rho+\lambda)$ is indecomposable over $G_1$ for every $\lambda \in X_1(T)$.
\end{theorem}

It is still not clear whether the tilting module conjecture will prove true in all characteristics, but it is undeniably a good guess.  Virtually nothing important about tilting modules (that we can think of) is characteristic dependent, and their Ext-vanishing properties and closure as a set under duality make them a natural home for the lifts of projective $G_1$-modules (which are also the injective $G_1$-modules).  It is almost hard to see a reason why this conjecture will fail to hold somewhere (though it might).  We view Theorem \ref{T:A} as adding to this sentiment, as the statement of the isomorphism appears to be so natural it suggests it should be true for all $p$.  It is also worth adding here that there are primes (infinitely many in fact) for which we do not yet know if the tilting module conjecture holds, but do know that the modules $T((p-1)\rho+\lambda) \otimes L((p-1)\rho-\lambda)$ are tilting (see \cite{BNPS}).  

This paper was motivated by recent work by Donkin \cite{D}, in which he proved, for $p$ good, that the module $\St \otimes \Dist{G_1}$ is tilting, where $G$ acts via the adjoint action on $\Dist{G_1}$.  One can see by character considerations that this tilting module is isomorphic to $p^{\rank(G)}$ copies of $\St \otimes \St \otimes \St$, and our main results came from looking for another natural way of describing this module, which was achieved by working in the category of $G_1T$-modules.  Specifically, with $\widehat{Q}_1(\lambda)$ denoting the $G_1T$-projective cover of $L(\lambda)$, we prove the following result (which is seen to be imply Theorem \ref{T:A} by repackaging the way the weights in this theorem are written, and applying a result from \cite{KN}).

\begin{theorem}\label{T:B}
There is an isomorphism of $G_1T$-modules
$$(\St \otimes \St)^{\oplus p^{\rank(G)}} \cong \bigoplus_{\lambda \in X_1} \left(\widehat{Q}_1(\lambda) \otimes L(\lambda)^* \right).$$
\end{theorem}

It is a basic result from the representation theory of the finite dimensional algebra $\Dist{G_1}$ that the dimensions of the modules in the isomorphism of this theorem are the same.  Moreover, since the maximal weight of  $\widehat{Q}_1(\lambda)$ is $2(p-1)\rho+w_0\lambda$ \cite[Lemma II.11.6]{rags}, it is straightforward to verify that the maximal weight appearing on each side of the isomorphism is $2(p-1)\rho$, and that it appears $p^{\rank(G)}$ times in both cases.  Nonetheless, showing that it was an isomorphism proved to be nontrivial, and our argument required a detour into the representation theory of the group scheme $G_1 \rtimes T$.  This group scheme, and the technique of working with an external semi-direct product group covering a subgroup scheme of $G$, may be of future use in this area.

\bigskip
\noindent \textbf{Acknowledgements} \, The author thanks Steve Donkin for several helpful correspondences and ideas that he shared as this work was being done.

\section{The representation theory of $G_1 \rtimes T$}

All notation follows that found in \cite{rags}.  In order to give statements over $G_1T$, it is necessary that we work over the slightly larger group scheme $G_1 \rtimes T$.  We begin with a general lemma about semidirect products in which the non-normal factor is a torus.

\begin{lemma}\label{L:torussummands}
Let $\mathcal{H}$ be an affine group scheme over $\Bbbk$, and $D$ an algebraic torus over $\Bbbk$ which acts on $\mathcal{H}$ by affine group scheme automorphisms.  If $V$ is a finite dimensional $\mathcal{H} \rtimes D$-module, then there is a decomposition
$$V = V_1 \oplus V_2 \oplus \cdots \oplus V_m$$
such that each $V_i$ is indecomposable over $\mathcal{H}$.  In particular, $V$ is indecomposable over $\mathcal{H} \rtimes D$ if and only if it is indecomposable over $\mathcal{H}$.
\end{lemma}

\begin{proof}
The representation is given by a group scheme homomorphism
$$\varphi: \mathcal{H} \rtimes D \rightarrow GL(V)$$
We have that
$$\text{Aut}_{\mathcal{H}}(V) = C_{GL(V)}(\varphi(\mathcal{H})),$$
and
$$\text{End}_{\mathcal{H}}(V) = C_{\mathfrak{gl}(V)}(\varphi(\mathcal{H})) = \text{Lie}(C_{GL(V)}(\varphi(\mathcal{H}))).$$
Since $D$ normalizes $\mathcal{H}$, $\varphi(D)$ normalizes $C_{GL(V)}(\varphi(\mathcal{H}))$.  Thus the product
$$C_{GL(V)}(\varphi(\mathcal{H}))\varphi(D)$$
is a closed connected subgroup of $GL(V)$ (being generated by closed connected subgroups).  Any set of orthogonal idempotents $e_1,e_2,\ldots,e_n$ in $\text{End}_{\mathcal{H}}(V)$ lie inside the Lie algebra of some maximal torus of $C_{GL(V)}(\varphi(\mathcal{H}))$ (because this algebraic group is a centralizer of a subgroup scheme), and therefore lie inside the Lie algebra some maximal torus
$$S \le C_{GL(V)}(\varphi(\mathcal{H}))\varphi(D).$$
By the conjugacy of maximal tori in an algebraic group over $\Bbbk$, there is some
$$g \in C_{GL(V)}(\varphi(\mathcal{H}))\varphi(D)$$
such that $g\varphi(D)g^{-1} \le S$.  Furthermore, since $\varphi(D)$ acts trivially on itself by conjugation, we may assume that $g \in C_{GL(V)}(\varphi(\mathcal{H}))$.  It now follows that
$$\{g^{-1}e_1g,\ldots,g^{-1}e_ng\}$$
is a set of orthogonal idempotents in $\text{End}_{\mathcal{H}}(V)$ that commute with $\varphi(D)$, hence lie in $\text{End}_{\mathcal{H} \rtimes D}(V)$.
\end{proof}

The inclusions of $T$ and $G_1$ into $G$ induce a surjective group scheme homomorphism
$$\pi_1: G_1 \rtimes T \rightarrow G_1T.$$
Note that $\pi_1$ is injective on both $G_1 \rtimes 1$ and on $1 \rtimes T$.  There is also the natural quotient homomorphism
$$\pi_2: G_1 \rtimes T \rightarrow T.$$

We see that $T_1 \rtimes T \subseteq G_1 \rtimes T$ is a maximal subgroup scheme of multiplicative type inside $G_1 \rtimes T$.  As a result, the character group of $G_1 \rtimes T$ identifies with $X_1(T) \times X(T)$.  We will show below that this set indexes the simple $G_1 \rtimes T$-modules.  However, when we refer to the ``$T$-character'' of a $G_1 \rtimes T$-module, we are considering it only as a module over $1 \rtimes T$.

Using $\pi_1$ and $\pi_2$ we can pull back each simple $G_1T$-module and simple $T$-module to $G_1 \rtimes T$.  As both homomorphisms are injective on $1 \rtimes T$, the character of $\pi_1^*(\widehat{L}_1(\lambda)) \otimes \pi_2^*(\mu)$ is just the product in $\mathbb{Z}[X(T)]$ of the character of $\widehat{L}_1(\lambda)$ with $e(\mu)$.  We record this fact for later use.

\begin{lemma}
For each $\lambda,\mu \in X(T)$ we have
$$\textup{ch}(\pi_1^*(\widehat{L}_1(\lambda)) \otimes \pi_2^*(\mu)) = \textup{ch}(\widehat{L}_1(\lambda))e(\mu),$$
and
$$\textup{ch}(\pi_1^*(\widehat{Q}_1(\lambda)) \otimes \pi_2^*(\mu)) = \textup{ch}(\widehat{Q}_1(\lambda))e(\mu).$$
\end{lemma}

\begin{lemma}
The $G_1 \rtimes T$-modules
$$\pi_1^*(\widehat{L}_1(\lambda)) \otimes \pi_2^*(\mu), \qquad \lambda \in X_1(T), \; \mu \in X(T),$$
form a complete set of pairwise non-isomorphic simple $G_1 \rtimes T$-modules.  The projective cover of $\pi_1^*(\widehat{L}_1(\lambda)) \otimes \pi_2^*(\mu)$ is given by $\pi_1^*(\widehat{Q}_1(\lambda)) \otimes \pi_2^*(\mu)$.
\end{lemma}

\begin{proof}
As noted earlier, every $\pi_1^*(\widehat{L}_1(\lambda))$ is simple over $G_1 \rtimes T$, and restricts over $G_1 \rtimes 1$ as $\widehat{L}_1(\lambda)$ restricts over $G_1 \le G_1T$.  Furthermore, it is clear that each $\pi_1^*(\widehat{L}_1(\lambda)) \otimes \pi_2^*(\mu)$ is a simple $G_1 \rtimes T$-module.  If
$$\pi_1^*(\widehat{L}_1(\lambda)) \otimes \pi_2^*(\mu) \cong \pi_1^*(\widehat{L}_1(\gamma)) \otimes \pi_2^*(\sigma)$$
for $\lambda,\gamma \in X_1(T)$, then
$$\pi_1^*(\widehat{L}_1(\lambda)) \otimes \pi_2^*(\mu-\sigma) \cong \pi_1^*(\widehat{L}_1(\gamma)).$$
Since $\lambda,\gamma \in X_1(T)$, the characters of $\widehat{L}_1(\lambda)$ and $\widehat{L}_1(\gamma)$ are characters of $G$-modules, hence are $W$-invariant.  It follows then that $\mu-\sigma$ must also be $W$-invariant, hence we conclude that $\mu-\sigma=0$, and that $\lambda =\gamma$.  Thus the set described is a set of pairwise non-isomorphic simple $G_1 \rtimes T$-modules.

Conversely, given a $G_1 \rtimes T$-module $V$, since $G_1 \rtimes 1$ is a normal subgroup we have
$$\textup{soc}_{G_1 \rtimes 1}(V) \subseteq V$$
as a $G_1 \rtimes T$-submodule.  If $V$ is simple, then it must be semisimple over $G_1 \rtimes 1$.  By Lemma \ref{L:torussummands}, it follows that $V$ is in fact simple over $G_1 \rtimes 1$, therefore isomorphic to some $\pi_1^*(\widehat{L}_1(\lambda))$.  There is a $G_1 \rtimes T$-module structure on
$$\textup{Hom}_{G_1 \rtimes 1}(\pi_1^*(\widehat{L}_1(\lambda)),V),$$
a one-dimensional vector space that is trivial as a module for $G_1 \rtimes 1$, thus the action factors through the map $\pi_2: G_1 \rtimes T \rightarrow T$.  It follows then that there is some $\mu$ such that as a $G_1 \rtimes T$-module,
$$\textup{Hom}_{G_1 \rtimes 1}(\pi_1^*(\widehat{L}_1(\lambda)),V) \cong \pi_2^*(\mu).$$
We then have that
$$V \cong \pi_1^*(\widehat{L}_1(\lambda)) \otimes \pi_2^*(\mu).$$
This proves that every simple $G_1 \rtimes T$-module is of the form given above.

Finally, the module $\pi_1^*(\widehat{Q}_1(\lambda)) \otimes \pi_2^*(\mu)$, being projective over $G_1 \rtimes 1$, is also projective over $G_1 \rtimes T$, and is easily seen to be the injective hull and projective cover of $\pi_1^*(\widehat{L}_1(\lambda)) \otimes \pi_2^*(\mu)$.
\end{proof}

We note that when working with $G_1T$-modules, the characters of the simple modules are linearly independent in $\mathbb{Z}[X(T)]$.  This is not true for $G_1 \rtimes T$-modules, since the characters of the one-dimensional modules $\pi_2^*(\mu)$ alone form a basis of $\mathbb{Z}[X(T)]$.  However, for a fixed simple $G_1$-module $L_1(\lambda)$, the characters of all simple $G_1 \rtimes T$-modules which are isomorphic to $L_1(\lambda)$ over $G_1$ are linearly independent.

\begin{lemma}\label{L:determinedbychar}
A semisimple $G_1 \rtimes T$-module is not determined by its $T$-character, but it is determined by the $T$-characters of each isotypic component over $G_1$.
\end{lemma}

The following result is familiar from the representation theory of $G_1T$.

\begin{lemma}\label{L:projdecomp}
Let $V$ be a projective $G_1 \rtimes T$-module.  Then there is a $G_1 \rtimes T$-module isomorphism
\begin{eqnarray*}
V & \cong &\bigoplus \left(\pi_1^*(\widehat{Q}_1(\lambda)) \otimes \textup{Hom}_{G_1}(\pi_1^*(\widehat{L}_1(\lambda)),V)\right)\\
& \cong & \bigoplus \left(\pi_1^*(\widehat{Q}_1(\lambda)) \otimes \textup{Hom}_{G_1}(\pi_1^*(\widehat{L}_1(\lambda)),V/\textup{rad}_{G_1}(V))\right)\\
\end{eqnarray*}
\end{lemma}

\begin{proof}
The projective summands appearing in $V$ are determined by the $G_1 \rtimes T$-socle of $V$, which is the same as the $G_1$-socle as shown above.  The $G_1 \rtimes T$ head and socle are isomorphic.  From this the result follows exactly as in the $G_1T$-case (see \cite[Proposition 4.1.3]{So}.
\end{proof}

\section{Decomposing the adjoint action}

The conjugation action of $G$ on itself stabilizes the normal subgroup scheme $G_1$, making both $\Bbbk[G_1]$ and $\Dist{G_1}$ into $G$-modules.  Let $B$ be the Borel subgroup containing $T$ (implicitly chosen by a choice of dominant weights).  The triangular decomposition
$$\Dist{G_1} \cong \Dist{U_1} \otimes \Dist{T_1} \otimes \Dist{U_1^+},$$
where $U_1$ (resp. $U_1^+$) is the Frobenius kernel of the unipotent radical of the Borel subgroup $B$ (resp. the opposite Borel subgroup $B^+$) \cite[Lemma 3.3]{rags}, is stable under the conjugation action of $T$.  We have
\begin{eqnarray*}
\text{ch}(\Dist{U_1}) & = & \text{ch}(\St)e(-(p-1)\rho),\\
\text{ch}(\Dist{T_1}) & = & p^{\rank(G)}e(0),\\
\text{ch}(\Dist{U_1^+}) & = & \text{ch}(\St)e((p-1)\rho)
\end{eqnarray*}
From this we see that
$$\text{ch}(\Dist{G_1}) = p^{\rank(G)}\text{ch}(\St \otimes \St).$$
We will now compute this character in another way.

Denote by $\delta$ the diagonal embedding
$$\delta: G \rightarrow G \times G.$$
There is a $G \times G$-action on $G$ given by
$$(g_1,g_2).h = g_1hg_2^{-1} \quad \forall g_1,g_2,h \in G(A),  \text{ and every commutative $\Bbbk$-algebra $A$}.$$
The conjugation action of $G$ on itself is just a restriction to $\delta(G)$ of this $G \times G$-action.  Now, under the action by $G \times G$, the normal subgroup scheme $G_1 \le G$ is stabilized by the subgroup schemes
$$(G_1 \times G_1) \le (G \times G), \quad \delta(G) \le (G \times G),$$
and thus by any subgroup schemes contained in these.  In particular, it is stabilized by subgroup schemes $G_1 \times 1$ and $\delta(T)$.

We observe that $G_1 \times 1$ is normal subgroup scheme of $G \times $G, hence there is a subgroup scheme
$$(G_1 \times 1)\delta(T) \le (G \times G),$$
and this is easily seen to be isomorphic to $G_1 \rtimes T$.  This gives $\Dist{G_1}$ the structure of a $G_1 \rtimes T$-module via this identification.
It is projective as a $G_1 \times 1$-module, hence under this isomorphism is projective over $G_1 \rtimes T$.

\begin{theorem}
By restriction, regard the $G$-module $L(\lambda)$ as a $T$-module.  Under the action specified above, there is an isomorphism of $G_1 \rtimes T$-modules
$$\Dist{G_1} \cong \bigoplus_{\lambda \in X_1(T)} \pi_1^*(\widehat{Q}_1(\lambda)) \otimes \pi_2^*(L(\lambda)^*).$$
\end{theorem}

\begin{proof}
By Lemma \ref{L:projdecomp}, this reduces to proving that as a $T$-module,
$$\textup{Hom}_{G_1 \rtimes 1}(\pi_1^*(\widehat{L}_1(\lambda)),\Dist{G_1}/\textup{rad}_{G_1}(\Dist{G_1})) \cong L(\lambda)^*$$
for each $\lambda \in X_1(T)$.

The Jacobson radical $J(\Dist{G_1})$, being a canonical two-sided ideal of $\Dist{G_1}$, is stablized by the action of both $G_1 \times G_1$ and by $\delta(G)$.  The argument by Cline, Parshall, and Scott in \cite[Theorem 1]{CPS} shows that as a $G$-module (via the action of $\delta(G)$), there is an isomorphism 
$$\Dist{G_1}/\textup{rad}_{G_1}(\Dist{G_1}) \cong \bigoplus_{\lambda \in X_1(T)} \left(L(\lambda) \otimes L(\lambda)^*\right),$$
and that the summand $L(\lambda) \otimes L(\lambda)^*$ is precisely the $L_1(\lambda)$-isotypic component of
$$\Dist{G_1}/\textup{rad}_{G_1}(\Dist{G_1})$$
as a left $\Dist{G_1}$-module.  Applying Lemma \ref{L:determinedbychar}, it follows that the $L_1(\lambda)$-isotypic component, as a $G_1 \rtimes T$-module, is isomorphic to
$$\pi_1^*(\widehat{L}_1(\lambda)) \otimes \pi_2^*(L(\lambda)^*),$$
completing the proof.
\end{proof}

\section{Proof of Main Results}

It now follows that there is an equality of characters
$$p^{\rank(G)}\text{ch}(\St \otimes \St) = \sum_{\lambda \in X_1} \text{ch}\left(\widehat{Q}_1(\lambda) \otimes L(\lambda)^* \right).$$
Since projective $G_1T$-modules are determined by their characters, this proves Theorem \ref{T:B}.

Kildetoft and Nakano showed that if $T((p-1)\rho+\lambda)$ is indecomposable over $G_1$ for every $\lambda \in X_1(T)$, then $\St \otimes L(\mu)$ is tilting for all $\mu \in X_1(T)$ \cite{KN}.  If this holds, then one sees in fact that given any $\gamma \in (p-1)\rho+X(T)_+$, the module
$$T(\gamma) \otimes L(\lambda)$$
is tilting.  This implies in particular that all modules of the form
$$T((p-1)\rho+\lambda) \otimes L((p-1)\rho-\lambda)$$
are tilting (when every $T((p-1)\rho+\lambda)$ is indecomposable over $G_1$).  Since tilting modules are determined by their characters, we have that Theorem \ref{T:B} implies Theorem \ref{T:A}.

\section{Final Remarks}

We conclude with an observation that the highest weight summands of the isomorphism contained in Theorem \ref{T:A} are the same in all characteristics, without knowing if the tilting module conjecture holds.

\begin{lemma}
If $\lambda \in X_1(T)$ and $\gamma \in (p-1)\rho+X(T)_+$, then $T(\gamma + \lambda) \mid T(\gamma) \otimes L(\lambda)$.
\end{lemma}

\begin{proof}
This is just a variation of a key argument used in \cite{P}.  Let $\mu \in X(T)_+$ be the dominant weight such that
$$\gamma = (p-1)\rho + \mu.$$
Consider the module
$$M =T(\mu) \otimes T(\lambda) \otimes T((p-1)\rho-\lambda) \otimes L(\lambda).$$
The highest weight of $M$ is $\gamma+\lambda$, occuring with multiplicity $1$.  Since $\St$ is a summand of $T(\lambda) \otimes T((p-1)\rho-\lambda)$ and of $T((p-1)\rho-\lambda) \otimes L(\lambda)$, then decomposing this quadruple tensor product in two different ways, we obtain
$$T(\mu+\lambda) \otimes \St$$
as a summand in one way, and
$$T((p-1)\rho+\mu) \otimes L(\lambda)$$
in another.  The first shows that $T(\gamma+\lambda)$ is a summand of $M$.  Recalling that $\gamma=(p-1)\rho+\mu$, the second shows that this summand must appear in a decomposition of
$$T(\gamma) \otimes L(\lambda).$$
\end{proof}

\begin{remark}
In comparison with the proof of Theorem \ref{T:A} in the previous section, note that this lemma holds for all primes.  That is, it does not depend on knowing if $T(\gamma) \otimes L(\lambda)$ is tilting.
\end{remark}

\begin{cor}
If $\lambda \in X_1(T)$, then $T(2(p-1)\rho) \mid (T((p-1)\rho+\lambda)) \otimes L((p-1)\rho-\lambda))$.
\end{cor}

This shows that for all primes, the summand $T(2(p-1)\rho)$ appears $p^{\rank(G)}$-times in the direct sums of modules appearing in the statement of Theorem \ref{T:A}.  Thus, if the isomorphism fails to hold somewhere, it is because of an imbalance of the other summands.

\providecommand{\bysame}{\leavevmode\hbox
to3em{\hrulefill}\thinspace}


\begin{thebibliography}{888888888}

\bibitem[\sf AMRW]{AMRW} P. Achar, S. Makisumi, S. Riche, G. Williamson, Koszul duality for Kac-Moody groups and characters of tilting modules, {\em J. Amer. Math. Soc.} {\sf 32} (2019), 261-310.

\bibitem[\sf BNPS]{BNPS} C. P. Bendel, D. K. Nakano, C. Pillen, P. Sobaje, On tensoring with the Steinberg representation, to appear in {\em Transformation Groups}.

\bibitem[\sf CPS]{CPS} E. Cline, B.J. Parshall, L. Scott, On the tensor product theorem for algebraic groups, {\em J. Algebra} {\sf 63} (1980), no. 1, 264-267.

\bibitem[\sf Don]{D} S. Donkin, A note on the adjoint action of a semisimple group on the coordinate algebra of an
infinitesimal subgroup, preprint, 2018.

\bibitem[\sf Jan]{rags} J. C. Jantzen, {\em Representations of
Algebraic Groups}, Second Edition, Mathematical Surveys and Monographs, Vol. 107, American Mathematical Society, Providence RI, 2003. 

\bibitem[\sf K]{K} T. Kildetoft, Decomposition of tensor products involving a Steinberg module, {\em Algebr. Represent. Theory} {\sf 20} (2017), no. 4, 951-975.

\bibitem[\sf KN]{KN} T. Kildetoft, D. K. Nakano, On good $(p,r)$ filtrations for rational $G$-modules,  {\em J. Algebra} {\sf 423}, (2015), 702-725. 

\bibitem[\sf P]{P} C. Pillen, Tensor products of modules with restricted highest weight, {\em Comm. Algebra} {\bf 21} (1993) 3647-3661.

\bibitem[\sf RW]{RW} S. Riche, G. Williamson, Tilting modules and the $p$-canonical basis,  {\em Ast{\'e}risque} 2018, no. 397, ix+184 pp.

\bibitem[\sf So]{So} P. Sobaje, On $(p,r)$-filtrations and tilting modules,  {\em Proc. Amer. Math. Soc.} {\sf 146} (2018), no. 5, pp. 1951-1961.

\end{thebibliography}
\end{document}